\documentclass[a4paper,11pt,reqno]{amsart}
\usepackage[latin1]{inputenc}
\usepackage{amssymb,amsmath,graphicx,amsthm}
\usepackage{hyperref}
\graphicspath{{Figs/}}
\usepackage[usenames,dvipsnames]{xcolor}

\newcommand{\TLT}{\mathcal{T}}
\newcommand{\si}{\sigma}
\newcommand{\eps}{\varepsilon}

\newcommand{\ip}{{Insertpoint}}
\newcommand{\rp}{{Removepoint}}

\newcommand{\grow}[1]{\left[{#1}\right]}
\newcommand{\ggrow}[1]{\left[{#1}\right]^*}

\newtheorem{thm}{Theorem}[section]
\newtheorem{prop}[thm]{Proposition}
\newtheorem{lemma}[thm]{Lemma}
\newtheorem{cor}[thm]{Corollary}

\theoremstyle{definition}
\newtheorem{defi}[thm]{Definition}
\newtheorem{rem}[thm]{Remark}

\DeclareMathOperator{\dleft}{left}
\DeclareMathOperator{\dtop}{top}
\DeclareMathOperator{\diag}{diag}

\author{Jean-Christophe Aval}
\address{LaBRI, Universit\'e Bordeaux 1, 351 cours de la lib\'eration, 33405 Talence
 cedex, FRANCE}
\email{aval@labri.fr}

\author{Adrien Boussicault}
 \address{LaBRI, Universit\'e Bordeaux 1, 351 cours de la lib\'eration, 33405 Talence
 cedex, FRANCE}
\email{boussicault@labri.fr}

\author{Philippe Nadeau}
\address{CNRS, Universit\'e Lyon 1, Institut Camille Jordan, 43 boulevard du 11 novembre 1918, 69622 Villeurbanne cedex, France.}
\email{nadeau@math.univ-lyon1.fr}

\thanks{All authors are supported by the ANR (PSYCO project -- ANR-11-JS02-001)}

\title[Tree-like tableaux]{Tree-like tableaux}

\begin{document}

\begin{abstract}
 
In this work we introduce and study \emph{tree-like tableaux}, which are certain fillings of Ferrers diagrams in simple bijection with permutation tableaux and alternative tableaux. We exhibit an elementary insertion procedure on our tableaux which gives a clear proof that tree-like tableaux of size $n$ are counted by $n!$, and which moreover respects most of the well-known statistics studied originally on alternative and permutation tableaux. Our insertion procedure allows to define in particular two simple new bijections between tree-like tableaux and permutations: the first one is conceived specifically to respect the generalized pattern 2-31, while the second one respects the underlying tree of a tree-like tableau.

\end{abstract}

\maketitle

\section*{Introduction}

Permutation tableaux and alternative tableaux are equivalent combinatorial objects that have been the focus of intense research in recent years. Originally introduced by Postnikov~\cite{Postnikov}, they were soon studied by numerous combinatorialists~\cite{Burstein, CorNad,SteinWil, Williams_Grassmann,Nad,Vien}. They also popped up surprisingly in order to get a combinatorial understanding of the equilibrium state of the PASEP model from statistical mechanics: this is the seminal work of Corteel and Williams, see~\cite{CorWil_Markov, CorWil_Tableaux, CorWil_Staircase}.

In this work we introduce and study tree-like tableaux (cf. Definition~\ref{defi:tlt}), which are objects in simple bijection with alternative tableaux. Indeed, our results have immediate reformulations in terms of alternative/permutation tableaux (see Proposition~\ref{prop:reformulations}). We chose to focus on these new tableaux for one main reason: they exhibit a natural tree structure (giving them their name: cf. Figure~\ref{fig:example_tlt}, right) more clearly than the alternative tableaux, and we use this structure in Section~\ref{sub:bijection2}. As is mentioned in this last section, the present work originated in fact in the study of trees.

 The main result of this work is Theorem~\ref{thm:insertion}:\medskip

 \emph{There is a simple bijective correspondence $\ip$ between 
 \begin{enumerate}
 \item tree-like tableaux of size $n$ together with an integer $i\in\{1,\ldots,n+1\}$, and 
 \item tree-like tableaux of size $n+1$.
 \end{enumerate}
 }
\medskip

A variation $\ip^*$  for symmetric tableaux is also defined and shares similar properties, see Theorem~\ref{thm:insertionB}. We prove that both $\ip$ and $\ip^*$ carry various statistics of tableaux in a straightforward manner: we obtain thus new easy proofs of formulas enumerating tableaux and symmetric tableaux (Section~\ref{sub:refined_enum}), as well as information on the average number of crossings and cells of tableaux (Section~\ref{sec:crossperm}).

An immediate corollary of Theorem~\ref{thm:insertion} is that tree-like tableaux of size $n$ are enumerated by $n!$, while Theorem~\ref{thm:insertionB} shows that symmetric tableaux of size $2n+1$ are enumerated by $2^nn!$. Several bijections between tableaux and permutations appeared already in the literature; the ones that seem essentially distinct are~\cite{SteinWil} and the two bijections from~\cite{CorNad}. All of them give automatically a correspondence as in Theorem~\ref{thm:insertion}, but none of them is as elementary as $\ip$. Conversely, it is clear that $\ip$ allows to define various bijections between permutations and tableaux. We will describe two of them here: the first one sends crossings to occurrences of the generalized pattern 2-31, while the second one preserves the binary trees naturally attached to permutations and tree-like tableaux.
\medskip

Let us give a brief outline of this work. Section~\ref{sec:defs} introduces numerous definitions and notations, and most notably the tree-like tableaux which are the central focus of this work. Section~\ref{sec:insertion} is the core section of this article: we introduce our main tool, the insertion $\ip$, and prove that it gives a $1$-to-$(n+1)$ correspondence between tableaux of size $n$ and $n+1$. We use it to give elementary proofs of refined enumeration formulas for tableaux. We also define a modified insertion $\ip^*$ for symmetric tableaux from which refined enumeration formulas are derived in a similar fashion. In Section~\ref{sec:crossperm} we keep using the insertions $\ip$ and $\ip^*$ to enumerate crossings and cells in tableaux. We also give a bijection between square symmetric tableaux and ordered partitions. In Section~\ref{sec:bijections} we define two bijections between tree-like tableaux and permutations, both based on $\ip$: the first one sends crossings to occurrences of the pattern 2-31, while the second one ``preserves trees'': it sends the tree structure of the tree-like tableau to a tree  naturally attached to the permutation (its increasing tree without its labels). 

\section{Definitions and Notation}
\label{sec:defs}

\subsection{Basic definitions}

 A \emph{Ferrers diagram} $F$ is a left aligned finite set of unit cells in $\mathbb{Z}^2$, in decreasing number from top to bottom, considered up to translation: see Figure~\ref{fig:example_diag}, left. The \emph{half-perimeter} of $F$ is the sum of its number of rows plus its number of columns; it is also equal to the number of \emph{boundary edges}, which are the edges found on the Southeast border of the diagram. We will also consider \emph{boundary cells}, which are the cells of $F$ with no other cells to their Southeast.

 There is a natural Southwest to Northeast order on boundary edges, as well as on boundary  cells. Moreover, by considering the Southeast corner of boundary cells, these corners are naturally intertwined with boundary edges: we will thus speak of a boundary cell being Southwest or Northeast of a boundary edge. Two cells are \emph{adjacent} if they share an edge. 

 {\em Ribbons:} Given two Ferrers diagrams $F_1\subseteq F_2$, we say that the set of cells $S=F_2-F_1$ (set-theoretic difference) is a \emph{ribbon} if it is connected (with respect to adjacency) and contains no $2\times 2$ square. In this case we say that $S$ can be added to $F_1$, or that it can be removed from $F_2$. Note that a removable ribbon from $F$ is equivalently a connected set $S$ of boundary cells of $F$, such that the Southwest-most cell of $S$ has no cell of $F$ below it, and the Northeast-most cell of $S$ has no cell of $F$ to its right.

\begin{figure}[!ht]
\begin{center}
  \includegraphics[width=\textwidth]{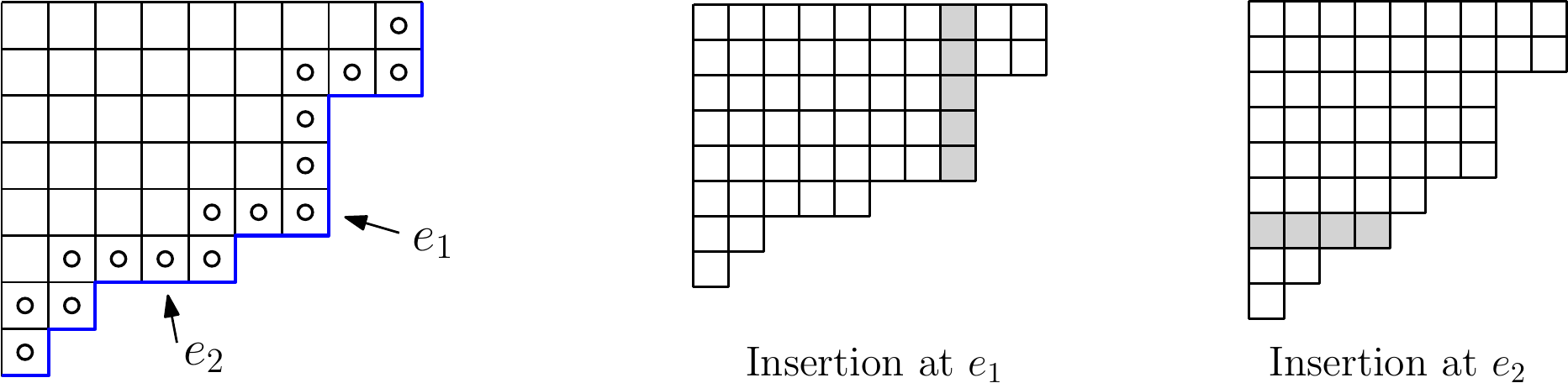}
\caption{A Ferrers diagram of half perimeter $17$ with its highlighted boundary cells and edges (left), and examples of column/row insertions at the boundary edges $e_1$ and $e_2$. \label{fig:example_diag}}
\end{center}
\end{figure}

  {\em Row/Column insertion:} Let $F$ be a Ferrers diagram and $e$ one of its boundary edges. If $e$ is at the end of a row $r$, we define the insertion of a column at $e$ to be the addition of a cell to $r$ and all rows above it; symmetrically, if $e$ is at the end of a column denoted by $c$, we can insert a row at $e$ by adding a cell to $c$ and all columns to its left; see Figure~\ref{fig:example_diag}, where the shaded cells of the figure are the added cells of the column or row. 
\medskip

\noindent{\bf Permutations and trees:}
 We consider \emph{permutations} $\si$ of $\{1,\ldots,n\}$, which are bijections from $\{1,\ldots,n\}$ to itself, and are counted by $n!$. We will represent permutations as words $\si_1\ldots\si_n$ of length $n$ where $\si_i=\si(i)$. A \emph{descent} is an index $i<n$ such that $\si_i>\si_{i+1}$. An occurrence of the pattern 2-31 in $\si$ is a pair $(i,j)$ of two indices such that $1\le i<j<n$ and $\si_{j+1}<\si_i<\si_j$.

A \emph{planar binary tree} is a rooted tree such that each vertex has either two ordered children or no child; vertices with no child are called \emph{leaves}, those of degree $2$ are called \emph{nodes}. The \emph{size} of a tree is its number of nodes; see Figure~\ref{fig:example_tlt} (right) for an example of tree of size $8$. 

\subsection{Tree-like tableaux}
\label{sub:tlt}

We can now define the main object of this work:
\begin{defi}[Tree-like tableau]
\label{defi:tlt}
A \emph{tree-like tableau} is a Ferrers diagram where each cell contains either $0$ or $1$ point (called respectively empty cell or pointed cell), with the following constraints:

 (1) the top left cell of the diagram contains a point, called the \emph{root point};

 (2) for every non-root pointed cell $c$, there exists either a pointed cell above $c$ in the same column, or a pointed cell to its left in the same row, {but not both}; 

 (3) every column and every row possesses at least one pointed cell.
\end{defi}
 
An example is shown on the left of Figure~\ref{fig:example_tlt}. 

\begin{rem}\label{rem:tree}
Condition (2) associates to each non-root point a unique other point above it or to its left. Now draw an edge between these two points for each non-root point, as well as an edge from every boundary edge to the closest point in its row or column: the result is a binary tree, where nodes and leaves correspond respectively to pointed cells and boundary edges. This is pictured in Figure~\ref{fig:example_tlt}, and explains the name {\em tree-like} given to our tableaux. We will come back to this tree structure with more detail in Section~\ref{sub:bijection2}.
\end{rem}

\begin{figure}[!ht]
\begin{center}
  \includegraphics[width=\textwidth]{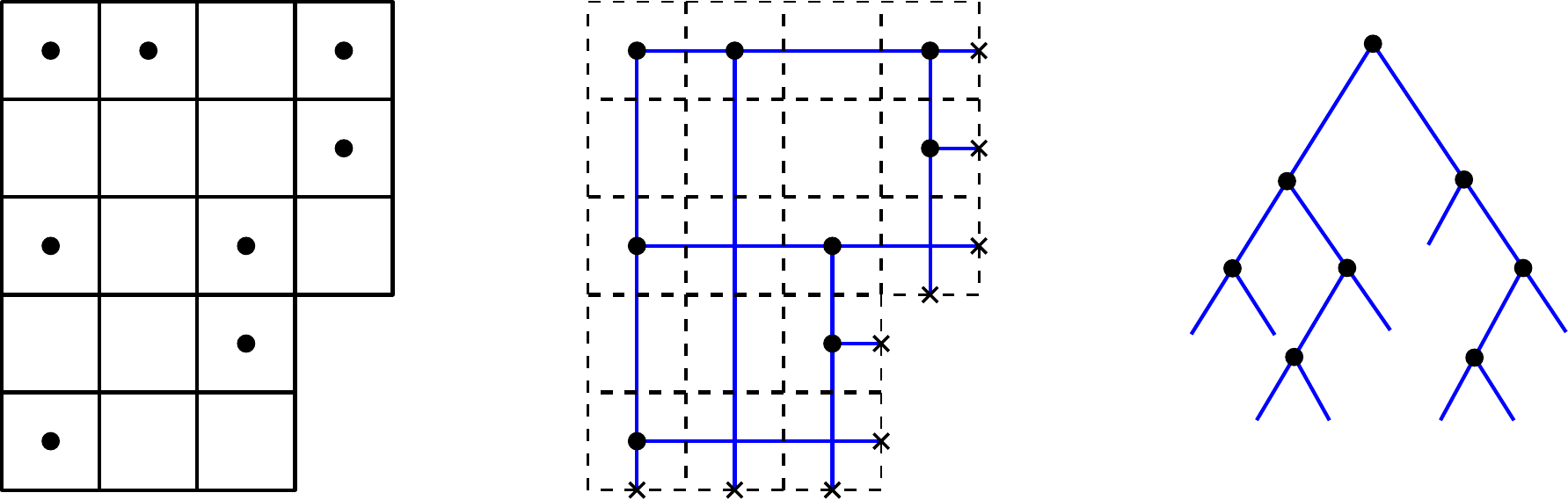}
\caption{A tree-like tableau (left) and the associated tree (right). \label{fig:example_tlt}}
\end{center}
\end{figure}

Let $T$ be a tree-like tableau. If the diagram of $T$ has half-perimeter $n+1$, then $T$ has exactly $n$ points: indeed Condition (2) associates to each row and column a unique point, except that the first row and column are both associated to the root-point. We let $n$ be the \emph{size} of $T$, and we denote by $\TLT_n$ the set of tree-like tableaux of size $n$. A \emph{crossing} of $T$ is an empty cell of $T$ with both a point above it and to its left; we let $cr(T)$ be the number of crossings of $T$. The {\em top points} ({respectively} {\em left points}) of $T$ are the non-root points appearing in the first row ({resp.} the first column) of its diagram. The tableau of Figure~\ref{fig:example_tlt} has $4$ crossings, $2$ top points and $2$ left points.
\medskip

\subsection{Alternative tableaux and permutation tableaux}

Tree-like ta\-bleaux are closely related to alternative tableaux~\cite{Nad,Vien} as follows: given a tree-like tableau, change every non-root point $p$ to an arrow which is oriented left (respectively up) if there is no point left of $p$ (resp. above $p$). This transforms the tableau into a \emph{packed} alternative tableau~\cite[Section 2.1.2]{Nad}, which is an alternative tableau with the maximal number of arrows for its half-perimeter. To obtain an alternative tableau, one simply deletes the first row and first column (empty rows and columns may then occur). 
Moreover, the shape is preserved: to any Ferrers diagram $F$, we associate (see Figure \ref{fig:diagrams}):
\begin{itemize}
\item a diagram $F'$ obtained by removing the Southwest-most boundary edge of $F$, and the cells of the left-most column,
\item a diagram $F''$ obtained by removing the Southwest-most and Northeast-most boundary edges of $F$, 
and the cells of the left-most column and of the top-most row.
\end{itemize}

\begin{figure}[!ht]
 \begin{center}
  \includegraphics[width=0.8\textwidth]{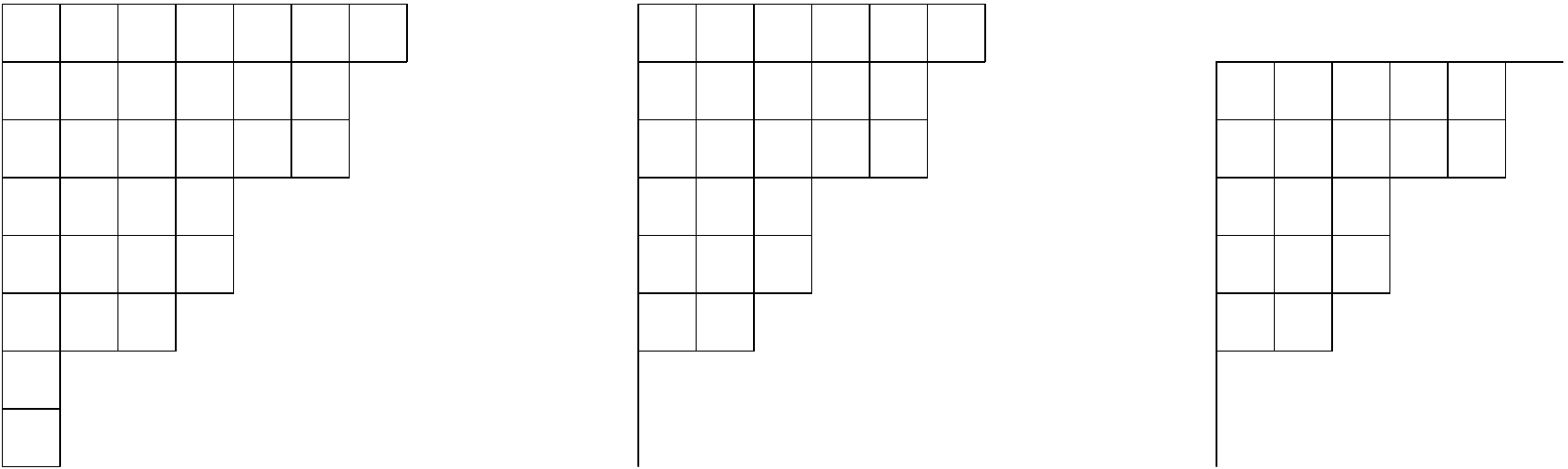}
  \caption{Diagrams $F$, $F'$ and $F''$. \label{fig:diagrams}}
 \end{center}
\end{figure}

We have the following correspondences, where we refer to~\cite{Nad} for definitions on alternative and permutation tableaux:

\begin{prop}
\label{prop:reformulations}
Let $n,i,j,k,\ell$ be nonnegative integers. There exist bijections between:
 
(1) Tree-like tableaux of half-perimeter $n+1$, with $i$ left points, $j$ top points, $k$ rows, $\ell$ crossings, and with shape a Ferrers diagram $F$.

(2) Permutation tableaux of half-perimeter $n$ with $i+1$ unrestricted rows, $j$ top ones, $k$ rows, $\ell$ superfluous ones, and with shape $F'$.

(3) Alternative tableaux of half-perimeter $n-1$ with $i$ free rows, $j$ free columns, $k-1$ rows, $\ell$ free cells, and with shape $F''$.
\end{prop}

\section{The fundamental insertion}
\label{sec:insertion}

This section is the core of this work. We describe a new way of inserting points in tree-like tableaux, shedding new light on numerous enumerative results on those tableaux.

\subsection{Main Result}
\label{sub:mainresult}

The key definition is the following one, which introduces a distinguished point in a tableau:
\begin{defi}[Special point]
\label{defi:sp} 
Let $T$ be a tree-like tableau. The \emph{special point} of $T$ is the Northeast-most point among those that occur at the bottom of a column.
\end{defi}
This is well-defined since the bottom row of $T$ necessarily has a pointed cell (Definition~\ref{defi:tlt}, (3)), which is then at the bottom of a column.
\medskip

\begin{defi}[$\ip$]
\label{defi:insertion}
  Let $T$ be a tableau of size $n$ and $e$ be one of its boundary edges. Let $T'$ be the tableau obtained by inserting a row (\emph{resp.} column) at $e$ and then pointing its rightmost (\emph{resp.} lowest) cell. Then we distinguish two cases:

(1) If $e$ is to the Northeast of the special point of $T$, then we simply define $\ip(T,e):=T'$;

(2) Otherwise, we add a ribbon starting just to the right of the new point of $T'$ and ending just below the special point of $T$. (If $e$ is the lower edge of the special cell of $T$, do nothing.) The result is a tableau $T''$, and define $\ip(T,e):=T''$.
\end{defi}

The result is a tree-like tableau of size $n+1$, since all three conditions of Definition~\ref{defi:tlt} are clearly satisfied. Examples of the two cases of $\ip$ are given in Figure~\ref{fig:insert}, while insertion at all possible edges of a given tableau is represented in Figure~\ref{fig:example_insert}. Cells from the inserted rows or columns are shaded, while those from added ribbons are marked with a cross. 

\begin{figure}[!ht]
\begin{center} 
 \includegraphics[width=\textwidth]{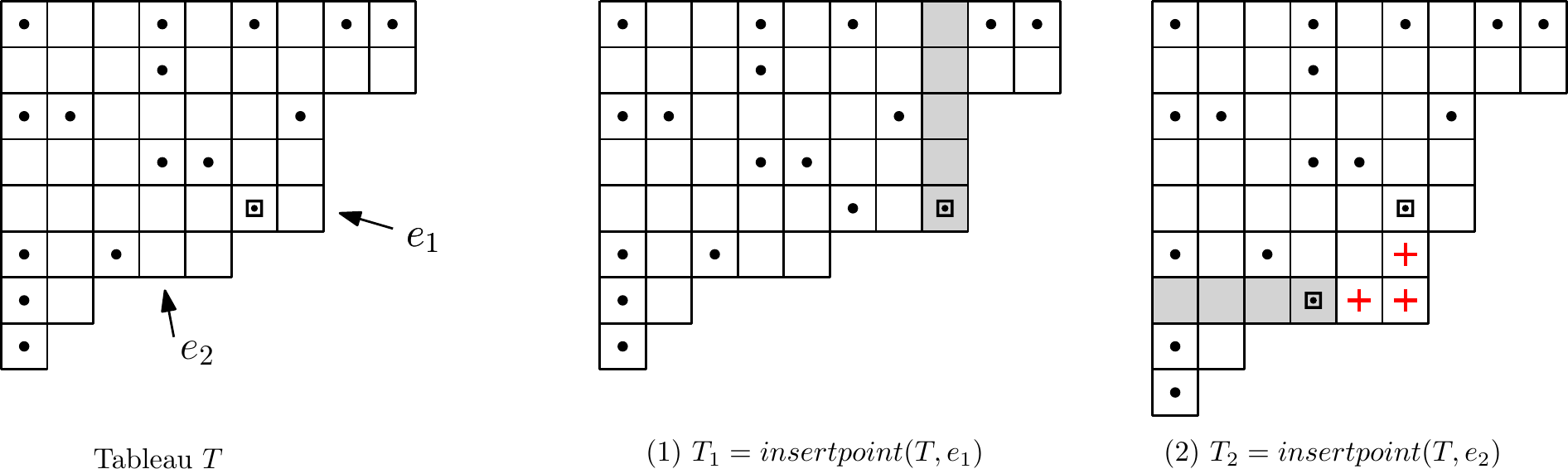}
\caption{The two cases in the definition of $\ip$. \label{fig:insert}}
\end{center}
\end{figure}

\begin{thm}
\label{thm:insertion}
For any $n\geq 1$, the insertion procedure $insertpoint$ is a bijection between:

(A) The set of pairs $(T,e)$ where $T\in \TLT_n$ and $e$ is one of the $n+1$ boundary edges of $T$, and

(B) The set $\TLT_{n+1}$.
\end{thm}

Before we give the proof, we need the following fundamental lemma:

\begin{lemma}
\label{lemma:fundamental}
 If $\ip(T,e):=T'$, then the special point of $T'$ is the new point added during the insertion.
\end{lemma}

\begin{proof}
 Notice that the new point $p'$ is at the bottom of a column of $T'$, so we must prove that the columns of $T'$ which are to the right of $p'$ do not have a bottom point.  If we are in case (1) of Definition~\ref{defi:insertion}, this is clear since these columns are the same as in $T$ and they are to the right of the special point $p$ of $T$. In case (2), all columns of $T'$ strictly to the right of $p$ and weakly to the left of $p'$ have a bottom cell coming from the added ribbon, and therefore contain no point. Since columns to the right of $p$ contain no bottom points either since $p$ is the special point of $T$, the proof is complete.
\end{proof}

\begin{figure}[!ht]
\begin{center}
  \includegraphics[width=\textwidth]{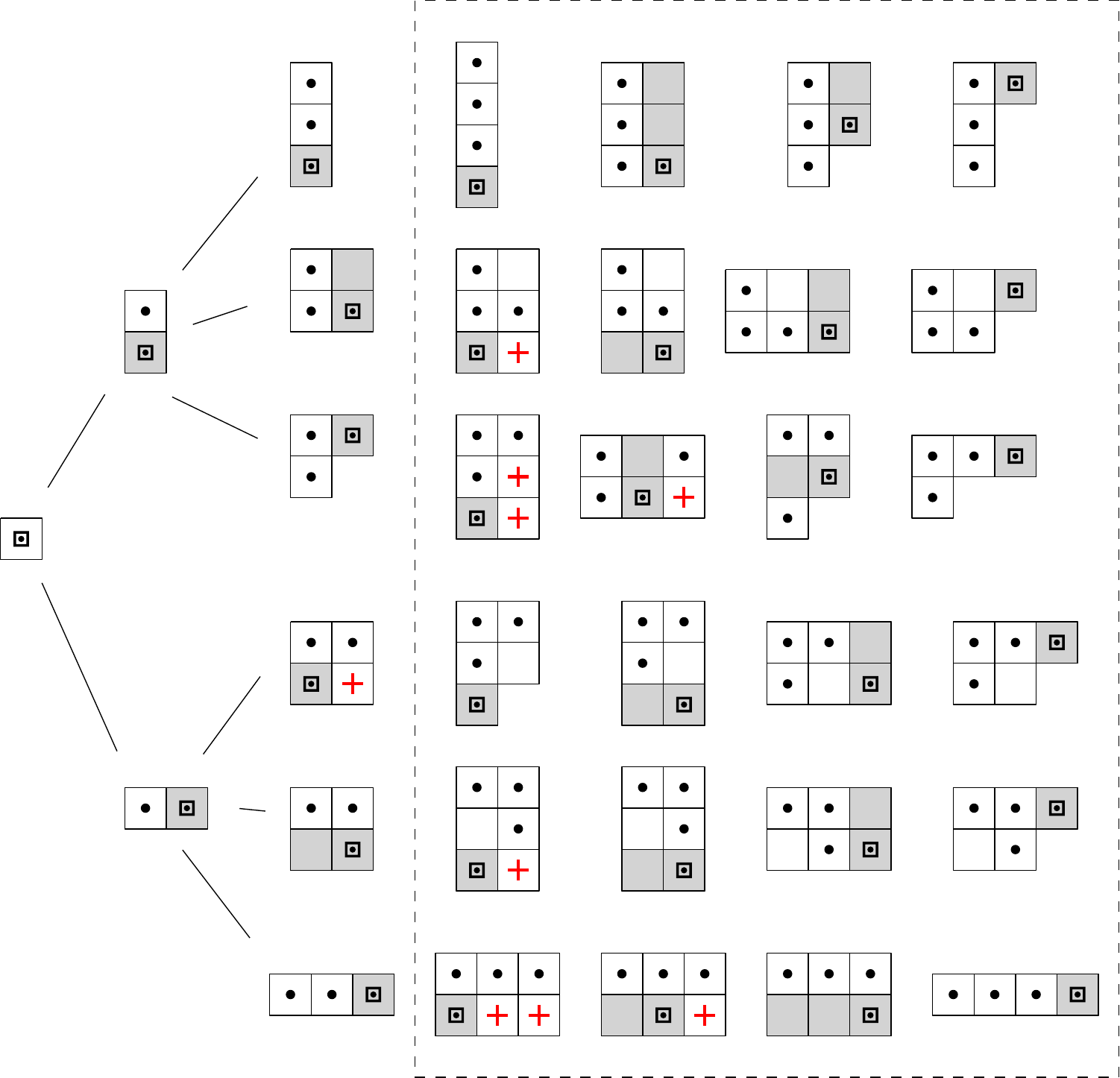}
\caption{Generating tree-like tableaux via $\ip$. \label{fig:generationtab4}}
\end{center}
\end{figure}

We can now give the proof of Theorem~\ref{thm:insertion}:

\begin{proof}[Proof of Theorem~\ref{thm:insertion}]
  We first define a function $\rp$;  we will then prove that it is the desired inverse of $\ip$. Let $T$ be a tableau of size $n+1$ and consider the cell $c$ containing the special point. In case there is a cell adjacent to the right of $c$, then follow the boundary cells to the Northeast of $c$ and let $c'$ be the first cell encountered which has a point (that cell exists since the last column possesses at least one point); we then remove the ribbon of empty cells comprised strictly between $c$ and $c'$. This leaves a Ferrers diagram since $c'$ is not the bottom cell of its column. Coming back to the general case, delete now the row or column which contains $c$ but no other points: let $T_1$ be the resulting tableau, and $e$ be the boundary edge of $T_1$ which is adjacent to $c$ in $T$. We define $\rp(T):=(T_1,e)$; $T_1$ has clearly size $n$ and $e$ is one of its boundary edges, and we claim $\rp$ is the desired inverse to the function $\ip$. 

It is clear that if $T'=\ip(T,e)$, then  $\rp(T')=(T,e)$: this is a consequence of Lemma~\ref{lemma:fundamental}. Let us now prove that $\ip\circ \rp$ is the identity on $\TLT_{n+1}$. So let $T\in \TLT_{n+1}$, with $c$ its special cell, and let $(T_1,e):=\rp(T)$. 
 Suppose first $c$ lies at the end of a row. In this case the special point of $T_1$ must be to the Southwest of $e$, therefore no ribbon will be added in $\ip(T_1,e)$ and this last tableau is thus clearly $T$. Now suppose there is a cell just to the right of $c$: in this case the cell $c'$ in the definition of $\rp$ contains the special point of $T_1$, since the removal of the ribbon will turn $c'$ into a bottom cell of a column. Now $e$ will be to the left of $c'$ in $T_1$, and so the application of $\ip$ will add the removed ribbon: in this case also $\ip(T_1,e)=T$, and this achieves the proof.
\end{proof} 

Since $|\TLT_1|=1$ we have the immediate corollary:

\begin{cor}
 $|\TLT_n|=n!$ for any $n\geq 1$.
\end{cor}

 So we have an elementary proof that tableaux of size $n$ are equinumerous with permutations of length $n$. In fact, many bijections can be deduced from $\ip$; we will describe two such bijections in Sections~\ref{sub:bijection1} and~\ref{sub:bijection2}.

\subsection{Symmetric tableaux}
In this section we consider \emph{symmetric tableaux}, i.e. tree-like tableaux which are invariant with respect to reflection through the main diagonal of their diagram; see an example in Figure~\ref{fig:BtoA}, left. Symmetric tree-like tableaux are in bijection with symmetric alternative tableaux from~\cite[Section 3.5]{Nad}, and ``type B permutation tableaux'' from ~\cite{LamWil}. 

The size of such a tableau is necessarily odd, and we denote by $\TLT_{2n+1}^{sym}$ the set of symmetric tableaux of size $2n+1$. $\TLT_{2n+1}^{sym}$ has cardinality $2^nn!$, as was shown in~\cite{LamWil,Nad}; we will give here a simple proof of this thanks to a modified insertion procedure.

\begin{figure}[!ht]
 \begin{center}
  \includegraphics[width=0.8\textwidth]{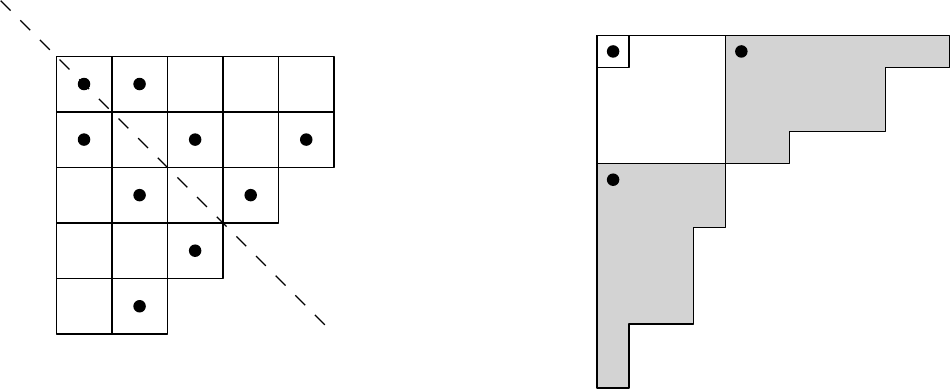}
  \caption{A symmetric tableau, and the embedding of tree-like tableaux in symmetric tableaux. \label{fig:BtoA}}
 \end{center}
\end{figure}

Note that given a tree-like tableau $T$ of size $n$, one can associate to it a symmetric tableau as follows: if $T$ has $k$ columns, then add on top of it a $k\times k$ square of cells, where only the top left cell is pointed; then add to the right of the square the reflected tableau $T^*$: see Figure~\ref{fig:BtoA}, right. In this way we embed naturally $\TLT_n$ in $\TLT_{2n+1}^{sym}$.

We now define a modified point insertion $\ip^*$ for symmetric tableaux. First let us call {\em $*$-special point} of a symmetric tableau the point at the bottom of a column which is Northeast-most \emph{among those that are Southwest of the diagonal}; we will call edges and cells below the diagonal the \emph{lower} edges and points. 

\begin{defi}[Symmetric insertion]
\label{defi:insertion_sym}
 Let $T\in\TLT_{2n+1}^{sym}$ and $(e,\eps)$ be a pair consisting of a lower boundary edge $e$ and $\eps\in\{+1,-1\}$. Define a first tableau $T'$ by inserting a row/column at $e$ with a point at the end, as well as the symmetric column/row. There are then three cases:

(1) If $\eps=+1$ and $e$ is Northeast of the $*$-special point, simply define $\ip^{*}(T,e,+1):=T'$.

(2) If $\eps=+1$ and $e$ is Southwest of the $*$-special point, add a ribbon to $T'$ between the new point (Southwest of the diagonal) and the $*$-special point of $T$ below the diagonal; add also the symmetric ribbon. If $T''$ is the resulting tableau, then define $\ip^{*}(T,e,+1):=T''$.

(3) If $\eps=-1$, add a ribbon in $T'$ between the two new points, and the resulting tableau is by definition $\ip^{*}(T,e,-1)$.\medskip
 \end{defi}

\begin{figure}[!ht]
 \begin{center}
  \includegraphics[width=\textwidth]{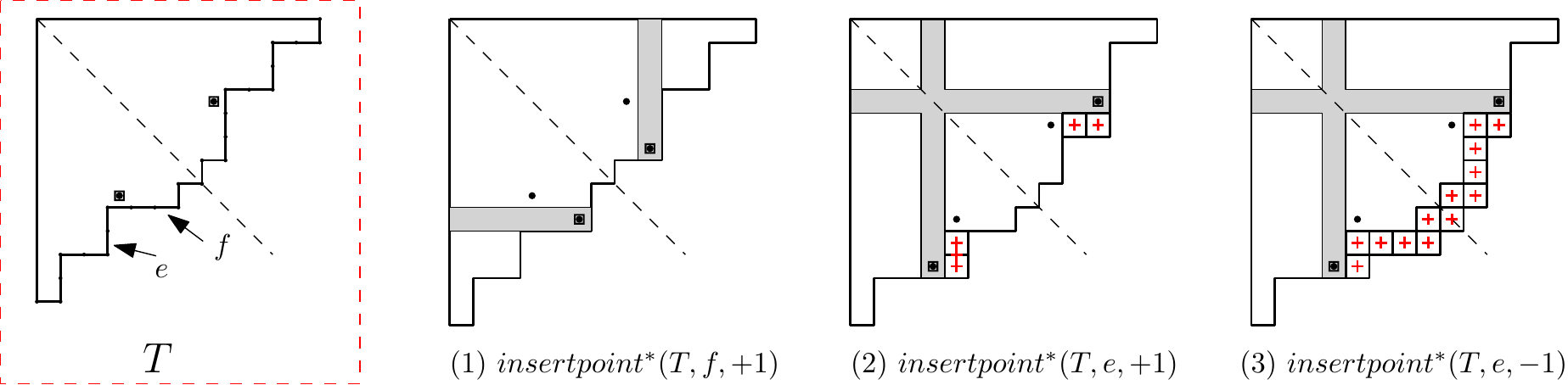}
  \caption{The three cases in the definition of $\ip^*$. \label{fig:syminsert}}
 \end{center}
\end{figure}

Examples of all $3$ cases are given in Figure~\ref{fig:syminsert}. It is easy to check that when only $\eps=+1$ is chosen during insertions, then the insertion produces precisely those symmetric tableaux given by the embedding of usual tree-like tableaux pictured in Figure~\ref{fig:BtoA}. Hence the symmetric insertion $\ip^{*}$ is a generalization of $\ip$.

\begin{thm}
\label{thm:insertionB}
 The procedure $\ip^{*}$ is a bijection between the set of triplets $(T,e,\eps)$ as in Definition~\ref{defi:insertion_sym} and $\TLT_{2n+3}^{sym}$.

\end{thm}

\begin{proof}
The key remark is the following natural generalization of Lemma~\ref{lemma:fundamental}: \emph{the new lower point inserted by $\ip^{*}$ is the $*$-special point of the resulting tableau}. In cases (1) and (2) of Definition~\ref{defi:insertion_sym}, the proof is the same as in Lemma~\ref{lemma:fundamental}, and this is clear in case (3).

The inverse of $\ip^*$ is defined as follows: given $T\in\TLT_{2n+3}^{sym}$, find its lower special point. If there is an empty cell to its right, follow the ribbon to the Northeast until the next pointed cell $c$. If $c$ is a lower cell, remove the ribbon of empty cells and its symmetric, and define $\eps:=1$; otherwise $c$ must be the cell symmetric to the lower special cell, and in this case remove from $T$ the ribbon of empty cells and let $\eps:=-1$. If there is no empty cell to the right of the lower special point, let $\eps:=+1$. For all cases, remove the row (\emph{resp.} column) which contains the special point and no other point, and let $e$ be the right (\emph{resp.} bottom) edge of the special point which remains in the resulting tableau $T'$. Then $T\mapsto (T',e,\eps)$ is the inverse of the insertion $\ip^{*}$: the proof is essentially the same as in the case of $\ip$ and is left to the reader.
\end{proof}

 We have the following immediate enumerative consequence:
 \begin{cor}
For  $n\geq 0$, 
 $$|\TLT_{2n+1}^{sym}|=2^n\,n!\,.$$
\end{cor}

\subsection{Refined enumeration}
\label{sub:refined_enum}

We now show how our insertion procedures give elementary proofs of some enumerative results on tableaux. 

Let $T_n(x,y)$ be the polynomial 
\[
 T_n(x,y)=\sum_{T\in \TLT_n} x^{\dleft(T)}y^{\dtop(T)},
\]
where $\dleft(T)$ and $\dtop(T)$ are respectively the number of left points and top points in $T$. When we insert a point in a tableau $T$ of size $n$, then we get an extra left (respectively right) point in the resulting tableau if the Southwest-most edge (resp. Northeast-most edges) is picked, while for other boundary edges the number of top and left points remains the same.

 This gives immediately the recurrence relation  $T_{n+1}(x,y)=(x+y+n-1)T_n(x,y)$ which together with $T_1(x,y)=1$ gives:
\begin{equation}
\label{eq:refined}
 T_n(x,y)=(x+y)(x+y+1)\cdots (x+y+n-2).
\end{equation}
 This formula was proved in~\cite{CorNad} and then bijectively in~\cite{Nad}; the proof just given is arguably the simplest one, and is bijective. 

We can also give a generalization of Formula~\eqref{eq:refined} to symmetric tableaux \cite{LamWil,CorKim}. Following~\cite[Section 5]{CorKim} --reformulated in terms of tree-like tableaux--, we define
\[
 T^{sym}_{2n+1}(x,y,z)=\sum_{T\in \TLT^{sym}_{2n+1}} x^{\dleft(T)}y^{\dtop^*(T)}z^{\diag(T)},
\]
where $\diag(T)$ is the number of crossings among the diagonal cells; for $\dtop^*(T)$, consider the northernmost non-root point $p$ in the first column, then the number of points on the row of $p$ is by definition $\dtop^*(T)$.

Let $T'=\ip^{*}(T,e,\eps)$ be as in Theorem~\ref{thm:insertionB}. One has $diag(T')=diag(T)+1$ when $\eps=-1$, and $diag(T')=diag(T)$ when $\eps=+1$. Also
$\dleft(T')=\dleft(T)+1$ when $e$ is the Southwest-most edge while $\dleft(T')=\dleft(T)$ for other choices of $e$.
 Finally, if the row $r$ considered in the definition of $\dtop^*(T)$ has its boundary edge $e'$ Southwest of the diagonal, then the insertion at $e=e'$ increases $\dtop^*(T)$ by one, while all other choices for $e$ leave $\dtop^*(T)$ invariant; if $r$ has its boundary edge $e'$ Northeast of the diagonal, then the column $c$ symmetric to $r$ ends below the diagonal at a boundary edge $e''$, and then the insertion at $e=e''$ increases $\dtop^*(T)$ by one while the other choices for $e$ leave $\dtop^*(T)$ invariant.

Putting things together we obtain the recurrence formula
 \[T^{sym}_{2n+3}(x,y,z)=(1+z)(x+y+n-1)T^{sym}_{2n+1}(x,y,z),\]
 from which it follows:
\begin{equation}
\label{eq:refinedB}
 T^{sym}_{2n+1}(x,y,z)=(1+z)^n(x+y)(x+y+1)\cdots (x+y+n-2).
\end{equation}

This proof is much simpler than any of the two proofs given in~\cite{CorKim}. Note also that $diag(T)=0$ means precisely that $T$ is of the form given on the right of Figure~\ref{fig:BtoA}, from which one gets
\[
  T_n(x,y)=T^{sym}_{2n+1}(x,y,0)
\]
and thus~\eqref{eq:refinedB} can be seen as an extension of~\eqref{eq:refined}.

\section{Enumeration of crossings and cells}
\label{sec:crossperm}

We denote by $\grow{T}$ the set of tableaux $\{\ip(T,e)\}$ where $e$ goes through all boundary edges of $T$. Theorem~\ref{thm:insertion} expresses that 
$\grow{T}$ has cardinality $n+1$ when $T$ has size $n$, and that we have the disjoint union
 \begin{equation}
  \label{eq:decompA}
  \TLT_{n+1}=\bigsqcup_{T\in\TLT_n}\grow{T}.
 \end{equation}
Similarly we denote by $\ggrow{T}$ the set of tableaux $\{\ip^*(T,e,\eps)\}$ where $e$ goes through lower boundary edges of $T$ and $\eps=\pm 1$. Theorem~\ref{thm:insertionB} expresses that 
$\ggrow{T}$ has cardinality $2(n+1)$ when $T$ has size $2n+1$, and that we have the disjoint union
\begin{equation}
  \label{eq:decompB}
\TLT^{sym}_{2n+3}=\bigsqcup_{T\in\TLT^{sym}_{2n+1}}\ggrow{T}
\end{equation}

Given a tableau $T=T^n$ of size $n$, there is a unique sequence of tableaux $\left(T^1,T^2,\ldots,T^{n-1},T^n=T\right)$ such that $T^i=\ip(T^{i-1},e^i)$ for certain boundary edges $e^i\in T^i$. Necessarily $T^i$ has size $i$ for all $i$ and in particular $T^1$ is the unique tableau of size $1$. We refer to this sequence of tableaux as the \emph{insertion history} of $T$. 

In this section we will analyze the mean number of crossings and cells that tree-like tableaux have, by analyzing how these quantities evolve through the insertions procedures of the previous section.

\subsection{Unrestricted tree-like tableaux}

Crossings of tree-like tableaux are an important statistic of these objects: They correspond to superfluous ones in permutation tableaux (cf.~Proposition~\ref{prop:reformulations}) and the work of Corteel and Williams~\cite{CorWil_Markov,CorWil_Tableaux} shows that this statistic is involved in the study of the PASEP.

 It turns out that crossings are particularly well-behaved with respect to our insertion procedure, as the following shows.

\begin{lemma}
\label{lem:crossings_ribbons}
Let $T\in \TLT_n$. The crossings of $T$ are the ribbon cells added in its insertion history.
\end{lemma}
\begin{proof}
 Consider $T'=\ip(T,e)$ for any tableau $T$ with $e$ as one of its boundary edges. Notice first that the empty cells of the inserted row or column are not crossings, and that the status of the empty cells of $T$ is left unchanged by this insertion. If a ribbon is inserted, then all cells of the ribbons are clearly crossings, since all rows and columns of $T$ contain at least one point. This shows that the crossings of $T'$ are those coming from $T$ plus the ribbon cells, which achieves the proof.
\end{proof}

Let $T$ be a tableau of size $n$. From the Southwest to the Northeast, we label its boundary edges $e_0(T)$, $\ldots$, $e_n(T)$  and its boundary cells $b_0(T)$, $\ldots$, $b_{n-1}(T)$. We have the following proposition whose easy proof is omitted:
\begin{prop}
\label{prop:bouh}
Let $T\in \TLT_n$ and $i\in\{0,\ldots,n\}$, and consider the tableau $T'=\ip(T,e_i(T))$. Then the special cell of $T'$ is $b_i(T')$. Moreover, if $b_k(T)$ is the special cell of $T$ (where $0\leq k\leq n-1$), then we have $cr(T')=cr(T)$ if $k\leq i$, while $cr(T')=cr(T)+(k-i)$ if $k>i$.
\end{prop}

\begin{figure}[!ht]
\begin{center} 
 \includegraphics[width=\textwidth]{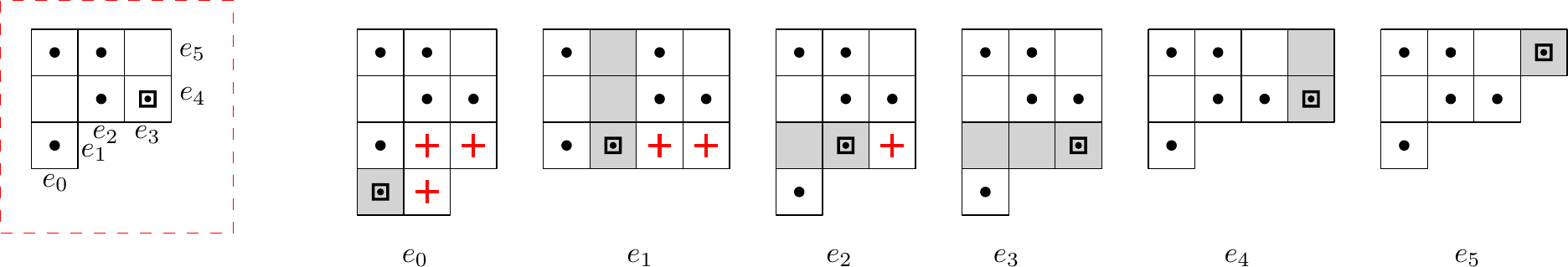}
\caption{A tableau of size $5$ and all $6$ possible point insertions. \label{fig:example_insert}}
\end{center}
\end{figure}

This is illustrated in Figure~\ref{fig:example_insert}, for which $k=3$ and $i$ goes from $0$ to $5$. A first consequence of the proposition is that given $k\in\{0,\ldots,n-1\}$, there are $n!/n=(n-1)!$ tableaux $T$ in $\TLT_n$ where the special cell of $T$ is $b_k(T)$. A second consequence is that given such a tableau $T$, the total number of ribbon cells added when constructing all tableaux in $\grow{T}$ is $1+2+\ldots+k=\binom{k+1}{2}$, and thus the total number of crossings in $\grow{T}$ is $(n+1)cr(T)+\binom{k+1}{2}$. Set $Cr_n=\sum_{T\in\TLT_n} cr(T)$, then we get from~\eqref{eq:decompA} that, for $n\geq 1$,:
\[
 Cr_{n+1}-(n+1)Cr_n=(n-1)!\times\sum_{k=0}^{n-1}\binom{k+1}{2}=(n-1)!\binom{n+1}{3}.
\]
If we let $X_n=Cr_n/n!$, we obtain simply $X_{n+1}-X_n=(n-1)/6$, from which we get:

\begin{prop}
\label{prop:totalcrossings}
The total number of crossings in $\TLT_n$ is given by $Cr_n=n!\times (n-1)(n-2)/12$.
\end{prop}

This can also be stated as: given the uniform distribution on $\TLT_n$, the expectation of $cr$ is given by $(n-1)(n-2)/12$. This was proved first in~\cite[Theorem 1]{CorHit} by a lengthy computation, which relied on the recursive construction of (permutation) tableaux obtained by adding the leftmost column. 

Now we want to enumerate all cells in the tableaux. Recall first that the Eulerian number $A(n,k)$ is defined as the number of permutations of length $n$ with $k-1$ ascents. In order to analyze the number of cells inserted during the insertion of a row or column, the following proposition is helpful.

\begin{prop}
\label{prop:eulerian}
The number of tree-like tableaux of size $n$ with $k$ rows is given by $A(n,k)$.
\end{prop}
\begin{proof} 
Suppose $T$ has $k$ rows. Then in $\grow{T}$ there are $k$ tableaux with $k$ rows and $n+1-k$ tableaux with $k+1$ rows, which correspond respectively to a column and a row insertion. From this one deduces that $A(n+1,k)=kA(n,k)+(n+2-k)A(n,k-1)$; this is the familiar recursion followed by Eulerian numbers.
\end{proof}
Introducing the Eulerian polynomial $A_n(t)=\sum_{k=1}^nA(n,k)t^k$, we have 
\[
 A_{n+1}(t)=(n+1)tA_n(t)+t(1-t)A_n'(t),
\]
with initial condition $A_0(t)=1$. If we differentiate this equation twice, and plug in $t=1$ in each case, one obtains equations for $A'_n(1)$ and $A''_n(1)$ which give (this is well-known):
\[
 A_n'(1)=\frac{(n+1)!}{2}\quad\text{and}\quad A''_n(1)=\frac{(n+1)!(3n-2)}{12},
\]
valid for $n\geq 1$ and $n\geq 2$ respectively.

\begin{prop}
\label{prop:nbcellsA}
The average number of cells in a tree-like tableau of size $n$ is  
\begin{equation}
 Y_n=\frac{(n+1)(5n+6)}{24}
\end{equation}
\end{prop}

\begin{proof}
 Suppose $T\in\TLT_n$ has $k$ rows, and thus $n+1-k$ columns, and let $ncr(T)$ be its number of \emph{non-crossing} cells. Note that if the edge $e$ is at the end of the $j$th row or column, then $\ip(T,e)$ has $ncr(T)+j$ non-crossing cells since we inserted a column or row with $j$ cells. Hence we have that $ncr(\grow{T})$ is equal to $(n+1)ncr(T)+k(k+1)/2+(n+1-k)(n+2-k)/2$; by summing over all tableaux $T\in\TLT_n$ we get for $n\geq 2$:
\begin{align*}
 ncr(\TLT_{n+1})=&(n+1)ncr(\TLT_n) + \sum_{k=1}^nA(n,k) \! \left( \! \frac{k(k \! + \! 1)}{2} \! + \! \frac{(n \! + \! 1 \! - \! k)(n \! + \! 2 \! - \! k)}{2} \! \right)\\
               =&(n+1)ncr(\TLT_n)+\sum_{k=1}^nA(n,k) k(k+1)\\
               =&(n+1)ncr(\TLT_n)+A''_n(1)+2A_n'(1)\\
               =&(n+1)ncr(\TLT_n)+(n+1)!\frac{3n+10}{12}
\end{align*}
where we used the fact that $A(n,k)=A(n+1-k)$ in the second equality. Dividing both sides by $(n+1)!$ and setting 
$AvNcr(n):=ncr(\TLT_{n})/n!$ we get the equation:
\[
 AvNcr(n+1)=AvNcr(n)+\frac{3n+10}{12}\quad\text{for~ }n\geq 2.
\]

With the initial condition $AvNcr(2)=2$ this gives $AvNcr(n)=1/24(3n^2+17n+2)$, which added to the average number of crossings $(n-1)(n-2)/12$ gives the result.
\end{proof}

\subsection{Symmetric tree-like tableaux}

We will give here analogues of the results of the previous section for symmetric tableaux

\begin{prop}
For $n\geq 1$, the average number of crossings in symmetric tree-like tableaux of size $2n+1$ is given by
\begin{equation}
\label{eq:propcrB}
X^*_{2n+1} = \frac{2n^2+1}{6}.
\end{equation}
\end{prop}

\begin{proof}
 Assume $n\geq 1$. We denote by $Cr^*_{2n+1}$ the total number of crossings in all symmetric tableaux of size $2n+1$. Let $T$ be such a tableau, and let $i\in\{0,\ldots n-1\}$ be the position of its $*$-special point, and we wish to compute the number of crossings in $\ggrow{T}$. Then, as for unrestricted tableaux, the crossings added by the insertions $\ip^*$ with $\eps=1$ are counted by $2(1+\ldots+i)=i(i+1)$, while the insertions with $\eps=-1$ contribute $1+3+\ldots+(2n+1)=(n+1)^2$ crossings. 

Summing over $T$ we get
\begin{align}
Cr^*_{2n+3} &= 2(n+1)\,Cr^*_{2n+1} + 2^{n}(n-1)!\sum_{i=0}^{n-1}\left(i(i+1) + (n+1)^2\right)\\
            &= 2(n+1)\,Cr^*_{2n+1} + 2^{n}(n-1)! \left(\frac{(n+1)n(n-1)}{3} + n(n+1)^2)\right)
\end{align}
which together with $X^*_{2n+1}=\frac{Cr^*_{2n+1}}{2^nn!}$  leads to:
\begin{equation}
X^*_{2n+3} = X^*_{2n+1} + \frac{n-1}{6} + \frac{n+1}{2}.
\end{equation}
Since $X^*_3=1/2$ we get by summation~\eqref{eq:propcrB}. Note that $X_1^*=0$ while~\eqref{eq:propcrB} would give $1/6$.
\end{proof}

\begin{prop}
\label{prop:Bnk}
The numbers $B(n,k)$ of symmetric tableaux of size $2n+1$ with $k$ diagonal cells obey the following recursion:
\[
 B(n+1,k)=kB(n,k)+(n+1)B(n,k-1)+(n+3-k)B(n,k-2).
\]
with $B(0,1)=1$ and $B(0,k)=0$ if $k\neq 1$.
\end{prop}

\begin{proof} 
Suppose $T$ has $k$ diagonal cells; then in $\ggrow{T}$ there are:
\begin{itemize}
\item $k$ tableaux with $k$ diagonal cells (row insertion when $\varepsilon=+1$);
\item $(n+1-k)+k=n+1$ tableaux with $k+1$ diagonal cells ( row insertion when $\varepsilon=-1$ or  column insertion when $\varepsilon=+1$);
\item $n+1-k$ tableaux with $k+2$ diagonal cells (column insertion when $\eps=-1$).
\end{itemize}
  and from this one deduces the above recursion.
\end{proof}

Let $T$ be a symmetric tableau of size $2n+1$. 
Since $\ip^*$ increase by 1 the width and the height of symmetric tableaux, the diagonal of $T$ contains less that $n+1$ cells.
We deduce that $B(n,k)!=0$ if and only if $k \in [1,n+1]$.

Introducing the polynomial $B_n(t)=\sum_{k=1}^{n+1}B(n,k)t^k$ the previous recurrence relations become 
\[
 B_{n+1}(t)=(n+1)(t+t^2)B_n(t)+(t-t^3)B_n'(t),
\] 
where $B_0(t)=t$. If we differentiate this equation we get:
\begin{align*}
 B_{n+1}'&=(n+1)(2t+1)B_n+\left((n+1)(t+t^2)+ (1-3t^2)\right)B_n'+(t-t^3)B''_{n};\label{eq:1}\\
\end{align*}
Substituting $t=1$ in this equation one gets:  
\[
B_{n+1}'(1)=3(n+1)B_n(1)+2nB_n'(1);
\]
Introducing $E^{(1)}_n=B_n'(1)/(2^nn!)$ this gives $E^{(1)}_{n+1}=\frac{n}{n+1}E^{(1)}_n+\frac{3}{2}$, which is easily solved:

\begin{prop}
 For $n\geq 1$, the average number of diagonal cells on a symmetric tableau of size $2n+1$ is $E^{(1)}_n=\frac{3(n+1)}{4}$.
\end{prop}
%

Using these results we can then obtain the average number of cells in symmetric tableaux, giving 

\begin{prop}
\label{prop:nbcellsB}
The average number of cells in a symmetric tree-like tableau of size $2n+1$ is 
\begin{equation}
 Y^*_n=\frac{(10n+11)(n+1)}{12}
\end{equation}
for $n\ge1$.
\end{prop}
\begin{proof}
Let $T$ be a symmetric tableau of size $2n+1$ with $k$ diagonal cells. The number of non crossing cells in $T$ is noted $ncr(T)$; now we want to compute $ncr(\ggrow{T})$. The row insertions contribute $2\cdot 2\sum_{i=1}^{k}i=2k(k+1)$ cells, while the column insertions contribute $2\cdot \sum_{i=1}^{n+1-k}[2(k+i)+1]=2(n+1-k)(n+k+3)$ cells; therefore
\begin{align*}
  ncr(\ggrow{T}) &=2(n+1)ncr(T)+2\left[k(k+1)+(n+1-k)(n+k+3)\right]\\
                   &=2(n+1)ncr(T)+2((n+1)(n+3)-k)
\end{align*}

We now sum this over all tableaux of size $2n+1$ and get 
\begin{align*}
 ncr(\TLT^{sym}_{2n+3})&=2(n+1)ncr(\TLT^{sym}_{2n+1})+2\sum_{k=1}^{n+1}B(n,k)((n+1)(n+3)-k)\\
                         &=2(n+1)ncr(\TLT^{sym}_{2n+1})+2(n+1)(n+3)2^nn!-2B_n'(1)
\end{align*}

If we note $BvNcr(n)=ncr(\TLT^{sym}_{2n+1})/(2^nn!)$ we get:
\[
 BvNcr(n+1)=BvNcr(n)+(n+3)-\frac{3}{4}=BvNcr(n)+n+\frac{9}{4}
\]
This gives $BvNcr(n) = \frac{2n^2+7n+3}{4}$, for $n \ge 1$, which together with~the crossing numbers \eqref{eq:propcrB} achieves the proof.

\end{proof}

\subsection{Square symmetric tableaux and ordered partitions}

It does not seem that the numbers $B(n,k)$ defined in Proposition~\ref{prop:Bnk} have been studied in all generality. There exist ``Eulerian number of type B'' (A060187 in~\cite{oeis}) but these are different: as shown in~\cite{CJKL}, these numbers $E_B(n,k)$ count symmetric tableaux of size $2n+1$ such that $k$ is the sum of the number of non crossing diagonal cells and {\em half} the number of crossing diagonal cells.

There is a recursive way to compute the exponential generating functions $E_h(x)=\sum_nB(n,n+1-h)x^n/n!$  as follows: 
\begin{equation}
\label{eq:E_h}
 hE_h'(x)=(1-x)E_{h-1}'(x)-E_{h-1}(x)+(h-3)E_{h-2}(x)-xE_{h-2}'(x).
\end{equation}

This is a reformulation of the recurrence of Proposition~\ref{prop:Bnk}. The initial values are $E_{-1}(x)=0$ and $E_0(x)=\sum_{n\geq 0}B(n,n+1)x^n/n!$ which is given by
\begin{equation}
\label{eq:E0}
E_0(x)=\frac{1}{2-\exp(x)}
\end{equation} 

We will prove this bijectively, as a consequence of Theorem~\ref{th:squaresymm} below.
\medskip

Say that a symmetric tableau is {\em square} if its shape is a square Ferrers diagram. For a given size $2n+1$, these are clearly the only tableaux with the maximum number $n+1$ of diagonal cells, and they are thus counted by the numbers $B(n,n+1)$ whose exponential generating function is by definition $E_0(x)$. An \emph{ordered partition} of $\{1,\ldots,n\}$ with $k$ blocks is a partition in $k$ blocks together with a linear ordering of the blocks; the corresponding counting sequence is given in \cite[A000670]{oeis}. We will construct a bijection $\Xi$ from square tableaux to ordered partitions.

Here we will consider only the cells of $T$ that are strictly below the diagonal: this is enough to reconstruct the whole tableau by symmetry, since non-root diagonal cells are necessarily empty in symmetric tableaux. Call these restrictions {\em half tableaux}.

\begin{figure}[!ht]
\begin{center}
 \includegraphics[width=0.4\textwidth]{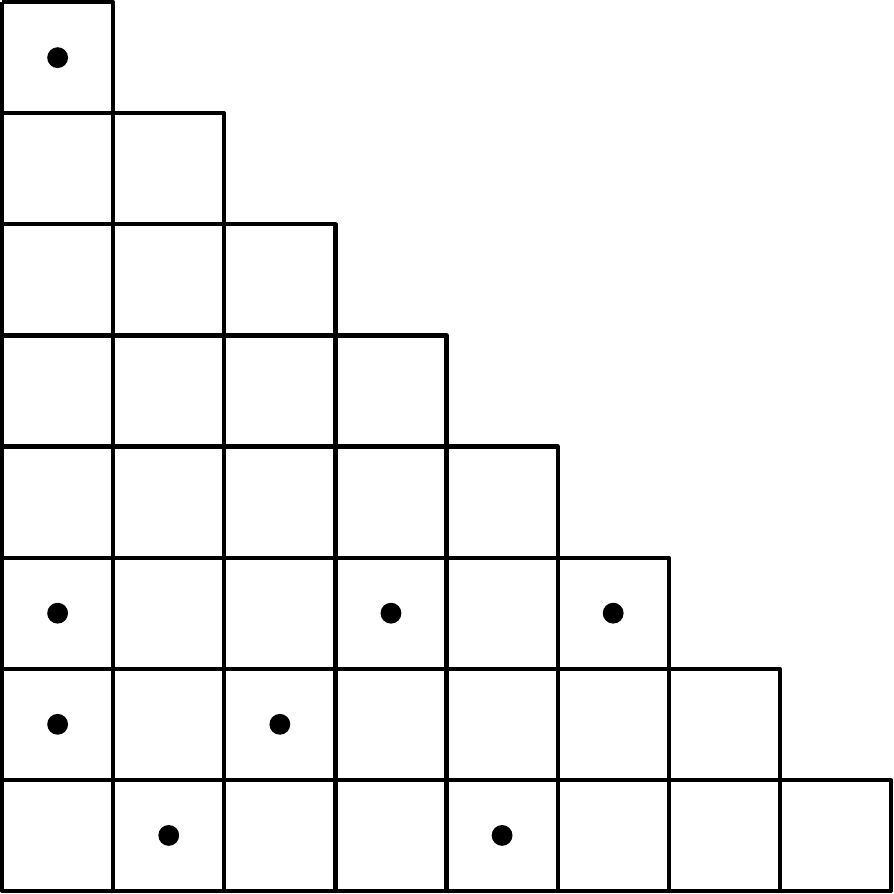}
\caption{A half square tableau of size $8$ \label{fig:half_square}}
\end{center}
\end{figure}

{\bf Bijection $\Xi$: } Let $T$ be a half tableau coming from a square symmetric tableau of size $2n+1$. Therefore $T$ consists of $n$ rows which have $1,2,\ldots,n$ cells from top to bottom, and has $n$ points in total; see Figure~\ref{fig:half_square}. 

Consider the last row of $T$, which has necessarily at least one point; let $B=\{i_1<\ldots<i_m\}$  be the positions where these points appear; here positions are labeled by $1,\ldots,n$ from left to right. Now notice that for $j>1$, column $i_j$ has no point beside the one in the last row, while row $i_j-1$ is empty. Delete all these rows and columns from $T$, together with the last row: after left- and bottom- justifying the remaining cells, the result is clearly a half tableau $T'$ of with $n-m$ points. 

By induction, $\Xi(T')$ is an ordered partition $\pi'=(B'_1,B'_2,\ldots,B'_\ell)$ of $\{1,\ldots,n-m\}$ where $\ell$ is the number of nonempty rows of $T'$; relabel the entries in the blocks of $\pi'$ by applying the only increasing bijection $\{1,\ldots,n-m\}$ to $\{1,\ldots,n\}\backslash \{i_1,\ldots,i_m\}$, and denote by $B_i$ the block $B'_i$ after this relabeling. Then $\Xi(T)$ is defined as the ordered partition $\pi:=(B_1,\ldots,B_\ell,B)$.

\begin{thm}
\label{th:squaresymm}
$\Xi$ is a bijection between:
\begin{itemize}
\item Square symmetric tableau of size $2n+1$, and
\item Ordered partitions of $[n]$. 
\end{itemize}
Moreover if $T$ has $j$ diagonal crossings and $\pi=\Xi(T)$, then $\pi$ has $j$ blocks.
 \end{thm}

As a corollary, we obtain the expression~\eqref{eq:E0} since $1/(2-\exp(x))$ is the exponential generating function of ordered partitions. 

An example of the bijection $\Xi$ is shown on Figure~\ref{fig:half_square_bij}. 
\begin{figure}[!ht]
\begin{center}
$$
	\begin{array}{ccccccc}
	\begin{array}{c}
		\includegraphics{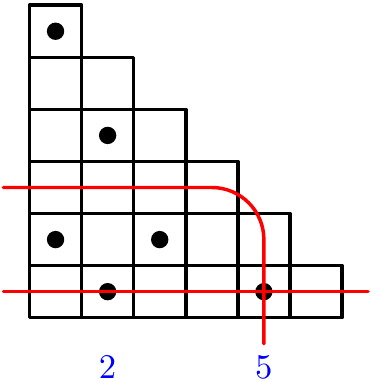}
	\end{array}
	&
	\hspace{-0.26cm}
	\rightarrow
	\hspace{-0.26cm}
	&
	\begin{array}{c}
		\includegraphics{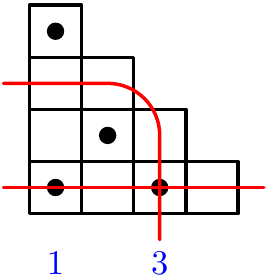}
	\end{array}
	&
	\hspace{-0.26cm}
	\rightarrow
	\hspace{-0.26cm}
	&
	\begin{array}{c}
		\includegraphics{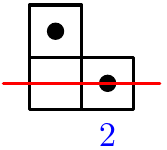}
	\end{array}
	&
	\hspace{-0.26cm}
	\rightarrow
	\hspace{-0.26cm}
	&
	\begin{array}{c}
		\includegraphics{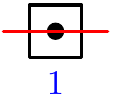}
	\end{array}\\
	(\{3\}, \{6\}, \{1,4\},\textcolor{blue}{ \{2,5\} })
	&
	\hspace{-0.26cm}
	\leftarrow
	\hspace{-0.26cm}
	&
	(\{2\}, \{4\}, \textcolor{blue}{ \{1,3\} })
	&
	\hspace{-0.26cm}
	\leftarrow
	\hspace{-0.26cm}
	&
	(\{1\}, \textcolor{blue}{ \{2\} })
	&
	\hspace{-0.26cm}
	\leftarrow
	\hspace{-0.26cm}
	&
	( \textcolor{blue}{ \{1\} })
	\end{array}
  $$
\caption{The bijection $\Xi$ between half square tableaux and ordered partitions \label{fig:half_square_bij}}
\end{center}
\end{figure}

\begin{proof}
The property about diagonal crossings and blocks is obvious for $\Phi$ by construction, since diagonal crossings correspond to nonempty rows in a half tableau.

We describe the reciprocal construction of $\Xi$, once again inductively. Assume $\Xi^{-1}$ is known for ordered partitions with $\ell$ blocks, an let $\pi$ have size $n$ and $\ell+1$ blocks $B_1,B_2,\ldots,B_{\ell+1}$. Assume $B_{\ell+1}$ has $m$ elements $\{i_1<\ldots<i_m\}$ and consider the partition $\pi'$ with $m$ blocks $B_1',B_2',\ldots,B'_{\ell}$ which are $B_1,B_2,\ldots, B_\ell$ relabeled on $\{1,\ldots,n-m\}$, so we can construct $T=\Xi^{-1}(\pi')$. In $T$, insert rows in positions $i_2-1,\ldots,i_m-1$ from top to bottom, and columns in positions $i_2,\ldots,i_m$ from left to right; to finish, add a bottom row with $n$ cells and points in positions $i_1,i_2,\ldots,i_m$. The result is the wanted (half) tableau $\Xi^{-1}(\pi)$. It is easily verified that this is indeed the inverse of $\Xi$.
\end{proof}

We end this section with two remarks. First notice the case $j=n$: the bijection $\Xi$ then restricts to a bijection between permutations of $\{1,\ldots,n\}$ and inversion arrays. It is tempting to consider half tableaux as generalized inversion arrays for ordered partitions and investigate which properties of permutations can be extended.

To finish, consider the numbers $B(n,n)$ which counts almost square symmetric tableaux, that is non square tableaux whose shape becomes square after adding a symmetric ribbon. The generating function $E_1(x)$ is $(1-x)/(2-\exp(x))$ by~\ref{eq:E_h}; thus $B(n,n)$ counts \emph{threshold graphs on $[n]$ with no isolated vertices} \cite[A053525]{oeis}. This can be also shown bijectively using the construction of  Proposition~\ref{prop:Bnk} in the case $k=n+2$, Theorem~\ref{th:squaresymm} and the formula in~\cite[Exercise 5.4.b.]{StanleyEnum2}

\section{Bijections with permutations}
\label{sec:bijections}

 It is pretty straightforward to derive bijections between permutations and tableaux from Theorem~\ref{thm:insertion}. Here we will single out two such bijections because of their specific properties. The first one sends crossings in a tableau to occurrences of the pattern 2-31 in a permutation, and the second one preserves a tree structure underlying both objects. 

\begin{rem}
Both bijections of this section can be extended with little effort to bijections between symmetric tableaux and {\em signed} permutations.
\end{rem}

\subsection{A first bijection which respects 2-31 patterns.}
\label{sub:bijection1}

As mentioned at the end of Section~\ref{sub:mainresult}, it is immediate to construct bijections from $\TLT_n$ to permutations using $\ip$. We will here define one with the goal of sending crossings of tableaux to occurrences of the pattern 2-31 in a permutation.

First, a tableau $T$ of size $n$ is naturally encoded by a list of integers $a(T)=(a_1(T),\ldots,a_n(T))$ satisfying $0\leq a_i(T)\leq i-1$. This is done as follows: let $T_1,T_2,T_3,\ldots,T_n=T$ be the  insertion history of $T$ . For $i\in\{2,\ldots,n\}$, we define $a_{i}(T)$ as the index $j$ such that $\ip(T_{i-1},e_j)=T_{i}$, using the labeling of boundary edges defined before Proposition~\ref{prop:bouh}.

We now give an algorithmic description of our first bijection $\Phi_1$. For $i=n,n-1,\ldots$ down to $i=1$, apply the following: consider the set $\{1,\ldots,n\}\backslash\{\si(i+1),\ldots,\si(n)\}=:\{x_0<x_1<\ldots<x_{i-1}\}$ arranged in increasing order and simply define $\si(i)=x_{a_i(T)}$; the function $\Phi$ is then defined by setting $\Phi_1(T)=\si$. For instance, if $a(T)=(0,1,0,3,1)$, one finds easily $\si=34152$.

\begin{thm}
\label{th:2bij}
 $\Phi_1$ is a bijection from $\TLT_n$ to permutations of length $n$. If $\si=\Phi_1(T)$, then $cr(T)$ is equal to the number of occurrences of 2-31 in $\si$.
\end{thm}

\begin{proof}

 First, it is clear that the construction is bijective. 
 
 Now  $i$ is a descent in $\si$ if and only if $a_{i}(T)\geq a_{i+1}(T)$. Moreover, in this case, this descent $i$ will give rise to exactly $a_{i}(T)-a_{i+1}(T)$ occurrences of the pattern 2-31 of the form $(\si(k),\si(i),\si(i+1))$ an occurrence of 2-31. 
 So we showed that the number of occurrences of 2-31 in $\si$ is given by $\sum_{i=1}^{n-1}max\left(a_{i}(T)-a_{i+1}(T),0\right)$.

But it is an easy reformulation of Proposition~\ref{prop:bouh} that this quantity is precisely $cr(T)$, which completes the proof.
 
\end{proof}

This bijection is much simpler than bijection II from~\cite{CorNad}, which was designed specifically to preserve the equivalent pattern 31-2.\medskip

\subsection{A second bijection which ``respects trees''.}
\label{sub:bijection2} 

Here we show that the tree structure of our tableaux can be naturally sent to the tree structure on permutations underlying their  representations as increasing trees.

\noindent{\bf From permutations to binary trees:} We define an \emph{increasing tree of size $n$} to be a binary tree of size $n$  where the $n$ nodes are labeled by all integers in $\{1,\ldots,n\}$ in such a way that the labels increase along the path from the root to any node. There is a well-known bijection with permutations: given an increasing tree, traverse its vertices in \emph{inorder}, which means recursively traverse the left subtree, then visit the root, then traverse the right subtree. By recording node labels in the order in which they are visited, one obtains the wanted permutation: see Figure~\ref{fig:treesfrompermtab} (left). If $\si$ is a permutation with associated increasing tree $inctree(\si)$, then we define $tree(\si)$ as the binary tree obtained by forgetting the labels in $inctree(\si)$.

\noindent{\bf From tree-like tableaux to binary trees:} We described this in Remark \ref{rem:tree}. It can also be obtained graphically by drawing two lines from every point of $T$, one down and one to the right, and stopping them at the boundary. We let $tree(T)$ be the binary tree thus constructed, see Figure~\ref{fig:treesfrompermtab} (right). Note that there is a natural identification between boundary edges of $T$ and leaves of $tree(T)$. 

\begin{figure}[!ht]
\begin{center} 
 \includegraphics[width=0.7\textwidth]{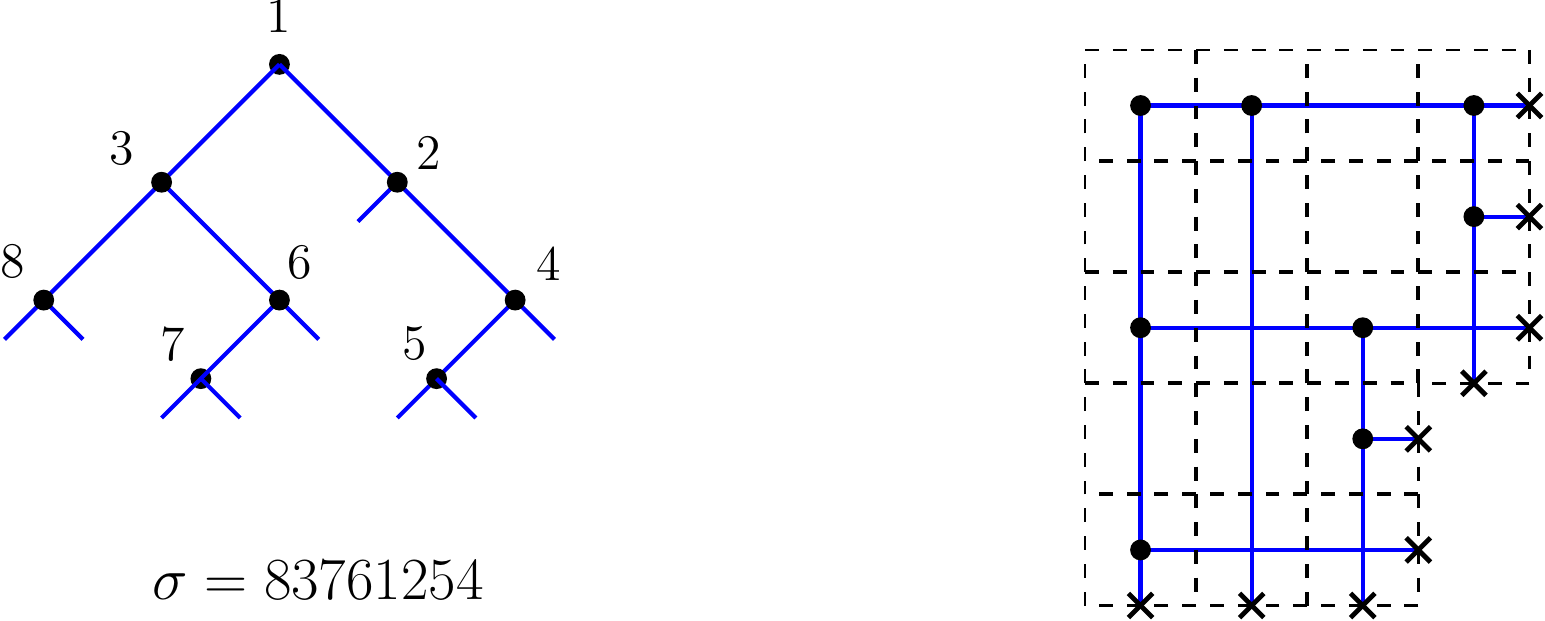}
\caption{The same binary tree arising from a permutation (left) and a tableau (right). \label{fig:treesfrompermtab}}
\end{center}
\end{figure}

Using $\ip$, we now define a bijection $\Phi_2$ between permutations and tree-like tableaux which preserves the binary trees attached to the objects. For this we proceed by induction on $n$.

 Let $\si$ be a permutation of size $n+1$, and $\tau$ be the permutation of size $n$ obtained by deleting $n+1$ in $\si$. By induction hypothesis, the tableau $T:=\Phi_2(\tau)$ is well defined and satisfies $tree(T)=tree(\tau)$. Define $L$ to be the leaf of $inctree(\tau)$ appearing in the inorder traversal at the position occupied by $n+1$ in $\si$: then $inctree(\si)$ is obtained by replacing $L$ by a node labeled $n+1$ with two leaves. Now $L$ corresponds naturally to a boundary edge $e_L$ in $\Phi_2(\tau)$, and we define $\Phi_2(\si)=\ip(\Phi_2(\tau),e_L$).

\begin{thm}
\label{th:1bij}
 Given $n\geq 1$, the function $\Phi_2$ is a bijection between permutations of length $n$ and tree-like tableaux of size $n$, satisfying $tree(\si)=tree(\Phi_2(\si))$.
\end{thm}

This is a simple consequence of the properties of $\ip$. The permutation $\si$ and the tree-like tableau $T$ from Figure~\ref{fig:treesfrompermtab} satisfy $\Phi_2(\si)=T$.
\medskip

\begin{rem}
The tree structure attached to tableaux is not new: first Burstein~\cite{Burstein} defined it on so-called \emph{bare tableaux}, which are essentially our tree-like tableaux minus a column. Then this tree structure was also studied by the third author in some detail~\cite[Section 4]{Nad}. The main difference is that, although the (unlabeled) tree structure is essentially the same, the labeling is quite different: here we have a quite simple bijection with increasing trees, while the labelings from the two aforementioned references involve some complicated increasing/decreasing conditions. The root of such complication can be traced to the fact that the boundary edges in~\cite{Burstein,Nad} were labeled independently of the structure of the tree, while here we use the tree to determine the labeling.
\end{rem}

\section{Further results and questions}
\label{sec:conclusion}

In this work we described a very simple insertion procedure $\ip$ which can be seen as a $1$-to-$(n+1)$ correspondence between the sets $\TLT_n$ and $\TLT_{n+1}$. We proved that from this simple seed one could produce automatically most of the enumerative results known on tableaux, as well as design bijections to permutations with various properties. Other enumeration results can also be proved with the same techniques: enumeration of tableaux according to the number of rows (this gives Eulerian numbers~\cite{Williams_Grassmann}), or the total number of cells. 

A further question would be to revisit the work of Corteel and Williams on the PASEP model from statistical mechanics (see~\cite{CorWil_Staircase, CorWil_Tableaux, CorWil_Markov}), which involves objects related to alternative tableaux. In particular, do their (weighted) staircase tableaux have recursive decompositions similar to those given here for tree-like tableaux ? An answer is given in \cite{staircase-ABD}.


As mentioned in the introduction, this work founds its origin in problems about trees, and not tableaux; we will here briefly describe such a problem. Suppose we draw the nodes of a plane binary tree as points in the center of unit cells of $\mathbb{Z}^2$, where the children of a node are drawn below and to the right of this node (as in the trees $tree(T)$ attached to a tree-like tableau $T$); we allow edges to cross outside of nodes. Let us call the drawing unambiguous if, when one deletes the edges of the tree, it is then possible to reconstruct them uniquely: one sees that this comes down essentially to condition (2) in Definition~\ref{defi:tlt}. 
We are led to the following definition: a \emph{ambiguous tree} is a tree-like tableau $T$ with rectangular shape. 
Such objects are investigated in \cite{UAT-ABBS} where several combinatorial results are obtained.

\subsubsection*{Acknowledgements}

The authors thank Valentin F\'eray for insightful comments. The second author is grateful to G\'erard H. E. Duchamp and \\Christophe Tollu for useful discussions during the CIP seminar at Paris 13 university, discussions which were the starting point of this work.

This research was driven by computer exploration using the open-source mathematical software \texttt{Sage} \cite{sage} and its algebraic combinatorics features developed by the \texttt{Sage-Combinat} community~\cite{Sage-Combinat}.


\bibliographystyle{alpha}
\bibliography{BiblioTreelikeTableaux}

\end{document}